\documentclass{svjour2}                    
\smartqed  
\usepackage{graphicx}
\usepackage{amssymb}
\journalname{}
\begin{document}

\title{Long time existence for the bosonic membrane in the light cone gauge
}


\author{Weiping Yan $^{a}$, Binlin Zhang $^{b}$
}


\institute{
\footnotesize $^a$ School of Mathematical Sciences, Xiamen University, Xiamen, 361005, China.\\
\email{yanwp@xmu.edu.cn}\\
\footnotesize $^b$ College of Mathematics and System Science, Shandong University of Science and Technology,
Qingdao, 266590, China\\
\email{zhangbinlin2012@163.com}}
\date{Received: date / Accepted: date}

\maketitle

\begin{abstract}
This paper aims to establish the well-posedness on time interval $[0,\varepsilon^{-\frac{1}{2}}T]$ of the classical initial problem for the bosonic membrane in the light cone gauge. Here $\varepsilon$ is the small parameter which measures the nonlinear effects.
The bosonic membrane is timelike submanifolds with vanishing mean curvature. Since the initial Riemannian metric may be degenerate,
the corresponding equation can be reduced to a quasi-linear degenerate hyperbolic system of second order with an area preserving constraint via a Hamiltonian reduction.

\keywords{Wave equations \and Smooth solutions  }
\end{abstract}

\section{Introduction and Main Results}
\label{intro}
Let $M$ be a compact $2$-dimensional manifold, $\textbf{R}^d$ be a $d$-dimensional Minkowski space and $\Sigma=\textbf{R}\times M$ be a $3$-dimensional submanifolds of Minkowski space $\textbf{R}^d$. $u=(u^{\mu})$ is an embedding from $\Sigma$ to $\textbf{R}^d$.
The critical points of the Nambu-Goto action
\begin{eqnarray}\label{E1-5}
\mathcal{S}=-\int_{\Sigma}\mu_g
\end{eqnarray}
give rise to submanifolds $\Sigma\subset\textbf{R}^d$ with vanishing mean curvature, where $\mu_g$ is a volume form is induced by a metric $g$ on $\Sigma$.
The Euler-Lagrange equation for functional $\mathcal{S}$ is
\begin{eqnarray*}
\sqrt{|g|}\square_gu^{\mu}=0,~~\mu=0,\ldots,d-1.
\end{eqnarray*}
If $\sqrt{|g|}\neq0$, it is equivalent to
\begin{eqnarray*}
(\delta_{\mu\nu}-g^{CD}\partial_Cu_{\mu}\partial_Du_{\nu})g^{AB}\partial_{A}\partial_Bu^{\nu}=0,
\end{eqnarray*}
where $\mu,\nu,\ldots=0,\ldots,d-1$ and $A,B,C,D,\ldots=0,1,2$ refer to coordinates on Minkowski space and $\Sigma$, respectively. In local coordinates, $g_{AB}$ is a Lorentzian metric  on $\Sigma$ with $\partial_t$ a timelike direction, which is expressed by $g_{AB}=\eta_{\mu\nu}\partial_Au^{\mu}\partial_Bu^{\nu}$, $\partial_Au^{\mu}=\frac{\partial u^{\mu}}{\partial\zeta^A}$. Here $\eta_{\mu\nu}$ is the Minkowski metric, which has the form $\eta_{\mu\nu}du^{\mu}du^{\nu}=-(du^0)^2+(du^1)^2+\ldots+(du^{d-1})^2$ in Cartesian coordinates $(u^{\mu})$,
$(\zeta^A)=(t,x^a)$ is a coordinates on $\Sigma$, $t$ is some global coordinate whose level sets $M_t$ foliate $\Sigma$, and $(u^a)$ is local coordinates on each $M$ with $a=1,2$.

The membrane system (or relativistic strings, see \cite{Hope}) arises in the context of membrane, supermembrane theories and higher-dimensional extensions of string theory. In Lorentzian geometric, since the critical point of the Nambu-Goto action gives rise to submanifolds with vanishing mean curvature, they are also called timelike minimal surface equations. They are one case of an important class of geometric evolution equations which is the Lorentzian analogue of the minimal submanifold equations. Since such equations possess plenty of geometric phenomenon and complicated structure (for example, they develop singularities in finite time, and degenerate and not strictly hyperbolic properties), much work is attracted in recent years. Lindblad \cite{Lindblad} and Brendle \cite{Brendle} proved the local and global well-posedness of timelike minimal surface equation with sufficiently small initial data in high dimension, respectively.
The case of general codimension and local well-posedness in the light cone gauge was studied by Allen, Andersson and Isenberg \cite{Allen} and Allen, Andersson and Restuccia \cite{Allen1}, respectively. He and Huang \cite{Huang2} obtained the existence of smooth solutions for the string and membrane equation in harmonic coordinates. Kong and his collaborators \cite{Kong2} obtained a representation formula of solution and presented many numerical evidence where singularity formation is prominent. Bellettini {\it et al.} \cite{Bel} showed that if the initial curve is a centrally symmetric convex curve and the initial velocity is zero, the string shrinks to a point in finite time. They noticed that it should be noted that the string does not become extinct there, but rather comes out of the singularity point, evolves back to its original shape and then periodically
afterwards. Nguyen and Tian \cite{Ngu} showed that timelike maximal cylinders in $\textbf{R}^{1+2}$ always develop singularities in finite time and that, infinitesimally at a generic singularity, their time slices are evolved by a rigid motion or a self-similar motion. They also proved a mild generalization in
non-flat backgrounds. He and Kong \cite{He} showed that there exist spherical symmetric solutions for relativistic membranes in the Schwarzschild spacetime. Huang and Kong \cite{Huang1} studied the motion of relativistic torus.

Before giving the motion equation of the canonical reduction of action (\ref{E1-5}) in the light cone gauge, we briefly recall the gauge fixing procedure, one can see \cite{Allen1} for more details. For gauge theories, one can see Dirac \cite{Dirac} for more details. We use the null coordinates $(u^+,u^-,u^m)$ and the volume form $\sqrt{w}$ on $M$ to specify the light cone gauge.
More precisely, $u^+$ and the corresponding conjugate momenta $p_+$ are
\begin{eqnarray*}
u^+=-p^0_{-}t,~~p_+=\frac{1}{2}(c\sqrt{w})^{-\frac{1}{2}}(p_mp^m+\gamma),
\end{eqnarray*}
$u^{-}$ and the corresponding conjugate momenta $p_-$ satisfy
\begin{eqnarray*}
\partial_au^{-}=-(c\sqrt{w})^{-\frac{1}{2}}p_m\partial_au^m,~~p_{-}=p^0_{-}\sqrt{w},
\end{eqnarray*}
where $p^0_{-}$ is a constant, $p_m=\frac{N}{\sqrt{\gamma}}(\partial_tu_{\mu}-N^a\partial_au_{\mu})$, $N^a=-g^{00}g^{0a}$, the metric $g^{AB}$ can be determined by performing the usual ADM decomposition of $g$. In order to eliminate the conjugate pairs $(x^-,p_{-})$ and $(x^+,p_{+})$, we provide the integrability condition
\begin{eqnarray*}
\int_{\mathcal{C}}p_m\sqrt{w}dx^m=0,~~for~all~closed~curves~\mathcal{C}~in~\Sigma.
\end{eqnarray*}
Furthermore, the light cone action can be obtained
\begin{eqnarray*}
\mathcal{S}=\int_{\Sigma}(p_m\partial_tu^m-cp_++\partial_t[p_-x^-]),
\end{eqnarray*}
the corresponding reduced Hamiltionian is
\begin{eqnarray*}
\mathcal{H}=\frac{\sqrt{w}}{4}\left(2\frac{p_mp^m}{\sqrt{w}^2}+\{u^m,u^n\}\{u_m,u_n\}\right)+p_m\{\Lambda,u^m\},
\end{eqnarray*}
where $\Lambda$ is determined by $N^a$, $\{\cdot\}$ denotes the Poisson bracket associated to a symplectic structure on $M$ in local coordinates.

Thus the equations of motion associated to the reduced Hamiltonian is the following second order system in light cone gauge (see also \cite{Hope})
\begin{eqnarray}\label{E1-1}
\partial_{tt}u^m=\{\{u^m,u^n\},u_n\},
\end{eqnarray}
with the initial data
\begin{eqnarray}\label{E1-0}
u(0)=u_0,~~\partial_tu(0)=u_1,
\end{eqnarray}
and the constraints
\begin{eqnarray}\label{E1-1R0}
\{\partial_tu^m,u_m\}=0.
\end{eqnarray}
Direct computation shows that the right hand side of (\ref{E1-1}) can be written in local coordinates as
\begin{eqnarray*}
\{\{u^m,u^n\},u_n\}=w^{-1}(\epsilon^{ac}\epsilon^{bd}\gamma(u)_{cd}\delta_{mn}-\epsilon^{ac}\epsilon^{bd}\partial_cu_m\partial_du_n)\partial_a\partial_bu^n+LOT,
\end{eqnarray*}
where $\gamma(u)_{cd}=\partial_cu^m\partial_du_m$ is the metric, $\epsilon^{ac}$ is the anti-symmetric symbol with two indices and LOT denotes the lower order terms. It is obviously that even the metric $\gamma(u)_{cd}$ is Riemannian, the first term in the right hand side of above inequality can cause the symbol to be degenerate. So
the hyperbolic property of equation (\ref{E1-1}) can not hold. As done in \cite{Allen1},
we modify the equation (\ref{E1-1}) by differentiating it with respect to $t$, then using the Jacobi identity and constraints (\ref{E1-1R0}) to obtain
\begin{eqnarray}\label{E1-1R00}
\partial_t\partial_{tt}u^m=\{\{\partial_tu^m,u^n\},u_n\}+2\{\{u^m,u^n\},\partial_tu_n\}.
\end{eqnarray}
Let
\begin{eqnarray}\label{E1-1R1}
\partial_{t}u^m=v^m,
\end{eqnarray}
system (\ref{E1-1R00}) can be written in local coordinates as
\begin{eqnarray}\label{E1-2}
\partial_{tt}v^m-\partial_a(\epsilon^{ac}\epsilon^{bd}\gamma(u)_{cd}w^{-1}\partial_bv^m)&-&2\partial_a(\epsilon^{ab}\epsilon^{cd}w^{-1}\partial_cu^m\partial_du^n\partial_bv_n)\nonumber\\
&-&\epsilon^{ab}w^{-\frac{1}{2}}\partial_c(\epsilon^{cd}w^{-\frac{1}{2}})(\gamma_{bd}\partial_av^m+2\partial_au^m\partial_bu^n\partial_dv_n)=0,~~~~~
\end{eqnarray}
where $\gamma(u)^{ab}=\epsilon^{ac}\epsilon^{bd}\gamma(u)_{cd}$, $a,b,c,d,\ldots=1,2$, $m,n,\ldots=1,\ldots,d-2$.

The corresponding initial data is
\begin{eqnarray}\label{E1-3}
u^m(0)=u^m_0,~~v^m(0)=v^m_0,~~\partial_tv^m(0)=v^m_1.
\end{eqnarray}
We need the following condition of initial data which make that
the solution $(u^m,v^m)$ of the modified system (\ref{E1-2}) is also the solution of equation ({\ref{E1-1}})
\begin{equation}
\label{E1-4} \left\{
\begin{array}{lll}
&&\{v_0^m,u_0^n\}\delta_{mn}=0,\\
&&v_1^m-\{\{u_0^m,u_0^n\},u_0^l\}\delta_{nl}=0.
\end{array}
\right.
\end{equation}
When initial Riemannian metric $\gamma(u_0)_{cd}$ is non-degenerate, second order system (\ref{E1-1R1})-(\ref{E1-2}) is a strictly hyperbolic system.
Allen, Anderson and Restuccia \cite{Allen1} showed that system (\ref{E1-1R1})-(\ref{E1-2}) has a unique solution $(u^m,v^n)\in\textbf{C}^1([0,T];\textbf{H}^k)\times\textbf{C}_T^k$ for $k\geq4$. Moreover, they got a blow up criterion.



The study of influence for degenerate metrics in nonlinear field equation is a very interesting problem, which can cause a degenerate nonlinear PDE.
Two of famous degenerate metrics are Schwarzschild metric and Kerr metric, which are two special solutions of nonlinear vacuum Einstein field equation. The corresponding stability problem is still an open problem (see \cite{C} for more detail). Anther goal of this paper is to deal with that initial Riemannian metric $\gamma(u_0)_{cd}$ is degenerate in LCG field equations,
then system (\ref{E1-1R1})-(\ref{E1-2}) is a degenerate hyperbolic system. More precisely,
we need to analyse the following quasi-linear toy model
\begin{eqnarray}\label{E1-6}
v_{tt}-\partial_a(\gamma^{ab}(t,x)\partial_bv)=g\left(t,x,\partial_av,\partial_{a}\partial_bv,\int_0^t\partial_av\right),
\end{eqnarray}
where $(t,x)\in[0,\frac{T}{\sqrt{\varepsilon}}]\times M$, $\gamma^{ab}(t,x)$ is a matrix which can be vanish at some point $(t,x)$, $f$ is a smooth bounded function and $g$ is a nonlinear term satisfies certain bounded conditions.
The Cauchy problem of degenerate nonlinear wave equation is a very interesting problem in hyperbolic differential equations. Colombini and Spagnolo \cite{Colom}
showed that the following Cauchy problem of one dimensional degenerate wave equation is not well posedness
\begin{eqnarray}\label{El}
u_{tt}-a(t)u_{xx}=0,
\end{eqnarray}
where $(t,x)\in[0,T]\times\textbf{R}$, $a(t)$ is a nonnegative smooth function and oscillates an infinite number of times. On the other hand, Colombini, De Giorgi and Spagnolo \cite{Colom1} showed that
any Cauchy problem as (\ref{El}) is well-posed in the space of the periodic real analytic functionals.
So it is interesting problem that what conditions of leading coefficients can make the degenerate nonlinear wave equation (even linear degenerate wave equation) being well-posed. We refer the paper of Han et.al.\cite{Han1,Han2} for Levi conditions and the paper of Nishitani \cite{Nis} for coefficients analysis conditions. In order to study the long time existence of a kind of (\ref{E1-6}), we require that the lower order coefficients satisfies certain growth conditions, i.e. Levi conditions.
To deal with loss of regularity, we construct a new Nash-Moser iteration scheme to solve denegerate wave equation (\ref{E1-2}).  This method have been employed by S. Klainerman \cite{Klainerman1,Klainerman2} to obtain the global existence and long time behavior for a class of evolution equations. Papageorgiou-R$\check{a}$dulescu-Repov$\check{s}$ \cite{PRR} studied nonlinear second order evolution inclusions with noncoercive viscosity term.

We denote by $\partial$ any time derivative $\partial_t$ or any coordinate derivative $D$.
For system (\ref{E1-2}),
we assume that
\begin{eqnarray}\label{E2-36}
\gamma(u_0)_{cd}=\gamma_0(x)\gamma(x)_{cd},
\end{eqnarray}
where $\gamma_0(x)$ is a smooth function, which satisfies
\begin{eqnarray}\label{E2-38}
0\leq\gamma_0(x)\leq1,~~|\partial_a\gamma_0(x)|\leq c_0\gamma_0(x).
\end{eqnarray}
One of simple example satisfying (\ref{E2-38}) is $r_0(x)=e^{-x}-1$.

The Riemannian metric $\gamma(x)_{cd}$ is non-degenerate and satisfies elliptic condition
\begin{eqnarray}\label{E2-37}
\gamma_1|\xi|^2\leq\gamma(x)_{cd}\xi_c\xi_d\leq\gamma_2|\xi|^2,~~\forall\xi\in M
\end{eqnarray}
 and the Levi condition
\begin{eqnarray}\label{E2-37R}
|\partial_a(\gamma(u_0)_{cd})\xi_c\xi_d|\leq c_1\gamma_0(x)|\xi|^2,~~\forall\xi\in M,
\end{eqnarray}
where $\gamma_1$, $\gamma_2$, $c_0$ and $c_1$ denote positive constants.

Even if the initial data make the Riemanian metric $\gamma_{cd}(u)$ degenerate, we can prove that
the well-posedness on time interval $[0,\varepsilon^{-\frac{1}{2}}T]$ is approached, $i.e.$ the existence of small amplitude smooth solution on $[0,\varepsilon^{-\frac{1}{2}}T]$.
More precisely, we have the following theorem.
\begin{theorem}
Assume that (\ref{E2-36})-(\ref{E2-37R}) holds. Let $k\geq2$, $T>0$ and $(\varepsilon^2u_0,\varepsilon^2v_0,\varepsilon^2v_1)\in\textbf{H}^k\times\textbf{H}^k\times\textbf{H}^{k-1}$.
Then, for a sufficient small constant $\varepsilon>0$, nonlinear second order system (\ref{E1-1}) with a small initial data (\ref{E1-0}) has a unique smooth solution
\begin{eqnarray*}
u^m(t,x)=u^m_0(x)+\int_0^tv^m(s,x)ds,~~(t,x)\in[0,\frac{T}{\sqrt{\varepsilon}}]\times M,
\end{eqnarray*}
which satisfies the constraint equation (\ref{E1-1R0}), where $v^m$ is a unique smooth solution of nonlinear degenerate hyperbolic system (\ref{E1-2}) with initial data (\ref{E1-3})--(\ref{E1-4}).
\end{theorem}

In Theorem 1, $\varepsilon$ is a sufficient small positive parameter. Since the smooth solution of (\ref{E1-0}) which is constructed in section 3, exists on $[0,\frac{T}{\sqrt{\varepsilon}}]\times M$, our result gives a large time behavior of solution for equation (\ref{E1-0}). Rescaling amplitude and time as $v^m(t,x)\mapsto\varepsilon^2 v^m(\sqrt{\varepsilon}t,x)$, one can see that the $\varepsilon$ measures the nonlinear terms of corresponding nonlinear wave system. One of important factors of the converge of Nash-Moser iteration scheme is the sufficient small property of
the parameter $\varepsilon$. In the other word, this iteration scheme only holds for a sufficient small $\varepsilon$. One can see section 3 for more details. It follows from the rescaling technique that the initial data $(v_0,v_1)$ is a small initial data, $i.e.$ $(\varepsilon^2v_0,\varepsilon^2 v_1)$. By the initial data relation (\ref{E1-4}), we know that the intial data $u_0$ is also a smalll initial data with size of $\varepsilon^2$.

The organization of this paper is as follows. In Section 2, we establish the existence of linear degenerate wave equation which  arises in the linearization of system (\ref{E1-2}). Section 3 is devoted to solving the initial value problem of the bosonic membrane equation (\ref{E1-1}) by means of a new Nash-Moser iteration scheme (For general Nash-Moser implict function theorem, one can see \cite{H,Moser,Nash,R,Yan}).  Furthermore, we obtain this solution is a unique solution in $\textbf{B}_{R,T}^{\bar{s}}$ (for $\textbf{B}_{R,T}^{\bar{s}}$, see (\ref{E2-35})).

\begin{acknowledgements}
The first author expresses his sincere thanks to Prof G. Tian for introducing him to the subject of minimal surface and his many discussions, and his sincere thanks to Prof Q. Han for his discussion on Nash-Moser iteration scheme, and thanks to prof S.J. Huang for his interest on this problem and some discussions.
The first author is supported by NSFC (No. 11771359), and the Fundamental Research Funds for the Central Universities (No. 20720190070 and 20720180009).
The second author is supported by NSFC (No. 11871199).
\end{acknowledgements}

\section{Energy estimates of the linearized equation}
This section is to discuss the linearized system for (\ref{E1-1}), which is a degenerate linear hyperbolic system. Throughout this section, we denote points in $\Sigma$ by $(t,x)$ with $t\in\textbf{R}$ and $x\in M$. For $d$-dimensional compact Riemannian manifold $M$ and smooth metric $\textbf{g}$,
the tangent space of $M$ is denoted by $T_xM$. $\Gamma^{\infty}(T^lM)$ is the set of all smooth tensor fields of type $(0,l)$, $\forall l\in[1,\infty]$. For each $x\in M$, $T^l_xM$ is an inner product space defined as follows. Let $\{e_1,\ldots,e_n\}$ be an orthonormal basis of $T_xM$. For any $\psi_1,\psi_2\in T^l_xM$, the inner product is given by
\begin{eqnarray*}
\langle\psi_1,\psi_2\rangle_{T_x^lM}=\sum_{i_1,i_2,\ldots,i_l=1}^n\psi_1(e_{i_1},\ldots,e_{i_l})\psi_2(e_{i_1},\ldots,e_{i_l}).
\end{eqnarray*}
In view of above inner product, $\Gamma^{\infty}(T^lM)$ are inner product spaces endowed with
\begin{eqnarray*}
\langle\psi_1,\psi_2\rangle_{\Gamma^{\infty}(T^lM)}=\int_M\langle\psi_1,\psi_2\rangle_{T^lM}\sqrt{w},~~\forall\psi_1,~\psi_2\in\Gamma^{\infty}(T^lM).
\end{eqnarray*}
We denote $\textbf{L}^2(M,\Gamma^{\infty}(T^lM))$ the completions of $\Gamma^{T^lM}$ in above inner product. The Sobolev space $\textbf{H}^l(M)$ is the completion of $\textbf{C}^{\infty}(M)$ with respect to the norm
\begin{eqnarray*}
\|u\|^2_{\textbf{H}^l}=\|u\|^2_{\textbf{L}^2}+\sum_{i=1}^l\|D^iu\|^2_{\textbf{L}^2(M,\Gamma^{\infty}(T^lM))}.
\end{eqnarray*}
In this paper, we consider the Sobolev space and function space of functions on $M$ in a fixed coordinate.
For any $l\in[1,\infty]$, $\textbf{H}^l$ denote the Sobolev space of functions on $M$ whose derivatives of up to order $l$ are square integrable. The corresponding norm is
\begin{eqnarray*}
\|u\|^2_{\textbf{L}^2}=\int_{M}|u|^2\sqrt{w},~~\|u\|_{\textbf{H}^1}^2=\|u\|^2_{\textbf{L}^2}+\int_{M}|Du|^2\sqrt{w},
\end{eqnarray*}
where $Du$ denote any coordinate derivative $\partial_au$.

We denote the spatially weighted Lebesgue spaces by $\textbf{L}^2_l(\mathbb{R}^n)$, which is equipped with the norm
\begin{eqnarray}\label{E2-0}
\|u\|_{\textbf{L}^2_l}=\int_{M}(1+|x|^2)^{\frac{l}{2}}|u|^2.
\end{eqnarray}
Then the Fourier transform is an isomorphism between $\textbf{L}^2_l(\mathbb{R}^n)$ and $\textbf{H}^l(\mathbb{R}^n)$.

Next we introduce some other spaces which is used in this paper. $\textbf{L}^{\infty}$ spaces and $\textbf{W}^{1,\infty}$ spaces equip with the norms
\begin{eqnarray*}
\|u\|_{\textbf{L}^{\infty}}=ess\sup_{M}|u|,~~\|u\|_{\textbf{W}^{1,\infty}}=\|u\|_{\textbf{L}^{\infty}}+\|Dv\|_{\textbf{L}^{\infty}},
\end{eqnarray*}
respectively.

$\textbf{C}^{r}([0,T];\textbf{H}^l)$ denotes the
function spaces with the norm
\begin{eqnarray*}
\|u\|_{\textbf{C}^{r}([0,T];\textbf{H}^l)}=\sup_{[0,T]}\sum_{i=0}^r\|\partial_{t}^iu\|_{\textbf{H}^l}.
\end{eqnarray*}

$\textbf{C}_T^{l}:=\cap_{i=0}^2\textbf{C}^i([0,T];\textbf{H}^{l-i})$ denotes function spaces with the spacetime norms
\begin{eqnarray*}
\|u\|_{l}^2=\sum_{i=0}^2\|\partial_{t}^iu\|^2_{\textbf{H}^{l-i}},~~|||u|||_{l,T}=\sup_{[0,T]}\|u(\cdot)\|_l.
\end{eqnarray*}
Generic constants are denoted by $c_0,c_1,\ldots$, their values may vary in the same
formula or in the same line.

We denote
\begin{eqnarray*}
\mathcal{B}(u)^a\partial_av^m&=&2\partial_a(\epsilon^{ab}\epsilon^{cd}w^{-1}\partial_cu^m\partial_du^n\partial_bv_n)\nonumber\\
&&-\epsilon^{ab}w^{-\frac{1}{2}}\partial_c(\epsilon^{cd}w^{-\frac{1}{2}})(\gamma_{bd}\partial_av^m+2\partial_au^m\partial_bu^n\partial_dv_n).
\end{eqnarray*}
Then equation (\ref{E1-2}) can be rewritten as
\begin{eqnarray}\label{E2-1}
\partial_{tt}v^m-\partial_a(\epsilon^{ac}\epsilon^{bd}\gamma(u)_{cd}\partial_bv^m)+\mathcal{B}(u)^a\partial_av^m=0.
\end{eqnarray}
It follows from (\ref{E1-1R1}) that
\begin{eqnarray}\label{E2-1R3}
\gamma(u)_{cd}=\gamma(u_0)_{cd}+\partial_cu_0^n\int_0^t\partial_dv_n+\partial_du_{0n}\int_0^t\partial_cv^n+\gamma(\int_0^tv)_{cd},
\end{eqnarray}
which implies that system (\ref{E2-1}) is equivalent to
\begin{eqnarray}\label{E2-1R1}
\partial_{tt}v^m&-&\partial_a(\epsilon^{ac}\epsilon^{bd}\gamma(u_0)_{cd}\partial_bv^m)-\partial_a(\epsilon^{ac}\epsilon^{bd}\partial_cu_0^m(\int_0^t\partial_dv_n)\partial_bv^m)\nonumber\\
&&-\partial_a(\epsilon^{ac}\epsilon^{bd}\partial_du_{0m}(\int_0^t\partial_cv^n)\partial_bv^m)-\partial_a(\epsilon^{ac}\epsilon^{bd}\gamma(\int_0^tv)_{cd}\partial_bv^m)\nonumber\\
&&+\mathcal{B}(u)^a\partial_av^m=0.
\end{eqnarray}
Since we assume that $\gamma(u_0)_{cd}$ is a degenerate Riemannian metric, system (\ref{E2-1R1}) is a degenerate quasi-linear wave system.
Linearizing quasi-linear degenerate wave system (\ref{E2-1R1}) around $v$ leads to the following linearized system with an external force
\begin{eqnarray}\label{E2-13}
\partial_{tt}h^m&-&\partial_a(\epsilon^{ac}\epsilon^{bd}\gamma(u_0)_{cd}\partial_bh^m)-\varepsilon\partial_a(\epsilon^{ac}\epsilon^{bd}\partial_cu_0^m(\int_0^t\partial_dv_n)\partial_bh^m)\nonumber\\
&&-\varepsilon\partial_a(\epsilon^{ac}\epsilon^{bd}\partial_cu_0^m(\int_0^t\partial_dh_n)\partial_bv^m)-\varepsilon\partial_a(\epsilon^{ac}\epsilon^{bd}\partial_du_{0m}(\int_0^t\partial_cv^n)\partial_bh^m)\nonumber\\
&&-\varepsilon\partial_a(\epsilon^{ac}\epsilon^{bd}\partial_du_{0m}(\int_0^t\partial_ch^n)\partial_bv^m)-\varepsilon\partial_a(\epsilon^{ac}\epsilon^{bd}\gamma(\int_0^tv)_{cd}\partial_bh^m)\nonumber\\
&&-\varepsilon\partial_a(\epsilon^{ac}\epsilon^{bd}(\int_0^t\partial_ch^n)(\int_0^t\partial_dv_n)\partial_bv^m)-\varepsilon\partial_a(\epsilon^{ac}\epsilon^{bd}(\int_0^t\partial_cv^n)(\int_0^t\partial_dh_n)\partial_bv^m)\nonumber\\
&&+\varepsilon\mathcal{B}(u)^a\partial_ah^m=g(t,x),
\end{eqnarray}
where
\begin{eqnarray*}
\mathcal{B}(u)^a\partial_ah^m&=&2\partial_a(\epsilon^{ab}\epsilon^{cd}w^{-1}(\int_0^t\partial_ch^m)(\int_0^t\partial_dv^n)\partial_bv_n)\nonumber\\
&&+2\partial_a(\epsilon^{ab}\epsilon^{cd}w^{-1}(\int_0^t\partial_cv^m)(\int_0^t\partial_dh^n)\partial_bv_n)\nonumber\\
&&+2\partial_a(\epsilon^{ab}\epsilon^{cd}w^{-1}(\int_0^t\partial_cv^m)(\int_0^t\partial_dv^n)\partial_bh_n)\nonumber\\
&&-\epsilon^{ab}w^{-\frac{1}{2}}\partial_c(\epsilon^{cd}w^{-\frac{1}{2}})(\gamma_{bd}\partial_ah^m+(\int_0^t\partial_bh^n)(\int_0^t\partial_dv_n)\partial_av^m)\nonumber\\
&&-\epsilon^{ab}w^{-\frac{1}{2}}\partial_c(\epsilon^{cd}w^{-\frac{1}{2}})((\int_0^t\partial_bv^n)(\int_0^t\partial_dh_n)\partial_av^m+2(\int_0^t\partial_bv^n)(\int_0^t\partial_ah^m)\partial_dv_n)\nonumber\\
&&-2\epsilon^{ab}w^{-\frac{1}{2}}\partial_c(\epsilon^{cd}w^{-\frac{1}{2}})((\int_0^t\partial_bh^n)(\int_0^t\partial_av^m)\partial_dv_n+(\int_0^t\partial_bv^n)(\int_0^t\partial_av^m)\partial_dh_n).
\end{eqnarray*}
This following result states the existence of linearized system with an external force $g(t,x)$ of the nonlinear bosonic membrane system (\ref{E2-1}). Here the external force $g(t,x)$ denotes the error term which is produced by carrying out the Nash-Moser iteration scheme in next section.

For $T>0$, $k\geq2$ and $0<R<1$, we define
\begin{eqnarray}\label{E2-35}
\textbf{B}_{R,T}^k:=\{u\in\textbf{C}_T^k:~|||u|||_{k,T}\leq R<1\}.
\end{eqnarray}
\begin{theorem}
Let $k\geq2$. Assume that (\ref{E2-36})-(\ref{E2-37R}) and $v\in\textbf{B}_{R,T}^k$ hold. Then for any initial data $(h_0,h_1)\in\textbf{H}^{k-1}\times\textbf{H}^{k-2}$, there exists an $\textbf{H}^{k-1}$ solution-$h(t,x)$ of (\ref{E2-13}) on $[0,T]\times M$. Moreover, for any $1\leq s\leq k-1$ and sufficient small $\varepsilon>0$, there holds
\begin{eqnarray}\label{E2-32}
|||h|||_{s,T}\leq c_3\left(|||h_0|||_{s,T}+|||h_1|||_{s,T}+|||g|||_{s,T}\right).
\end{eqnarray}
\end{theorem}
In order to explain our result more clearly,
we consider a toy model which is a two dimensional linear variable coefficent wave equation (not system) in local coordinate:
\begin{eqnarray}\label{E2-1R0}
\partial_{tt}h(t,x)-\varrho(x)\partial_a(\rho^{ab}(t,x)\partial_bh(t,x))&+&B(t,x)\partial_ah(t,x)\nonumber\\
&&-\varepsilon f(Du,Dv)\int_0^{t}\partial_ah(t,x)=g(t,x),~(t,x)\in[0,T]\times M,~~~~~~
\end{eqnarray}
with the initial data
\begin{eqnarray}\label{E2-2}
h(0)=h_0,~~\partial_th(0)=h_1,
\end{eqnarray}
where $\partial_a$ $(a,b=1,2)$ denotes coordinate derivative, $i.e.$ $\partial_1=\partial_{x_1}$ and $\partial_2=\partial_{x_2}$.

Assume that $f$ can be controlled by a constant in the norm $|||\cdot|||_{s,T}$ for some $s\in[1,k-1]$ when $(Du,Dv)\in\textbf{B}_{R,T}^k$. $\rho^{ab}(t,x)$ and $B(t,x)$ satisfy
\begin{eqnarray}\label{E2-15}
&&c_4|\xi|^2\leq\rho^{ab}\xi_a\xi_b\leq c_5|\xi|^2,~~\forall \xi\in M,\\
\label{E2-15R2}
&&0\leq\varrho(x)\leq1,~~x\in M,
\end{eqnarray}
 and the Levi condition
\begin{eqnarray}\label{E2-15R1}
|(B-\partial_a(\varrho\rho^{ab})\xi_a\xi_b|\leq c_6\varrho|\xi|^2,~~\xi\in M,
\end{eqnarray}
with some positive constants $c_4\leq c_5$.


\begin{theorem}
Assume that $k\geq2$ and (\ref{E2-15})-(\ref{E2-15R1}) holds. Then for any initial data $(h_0,h_1)\in\textbf{H}^{k-1}\times\textbf{H}^{k-2}$, there exists an $\textbf{H}^{k-1}$ solution-$h(t,x)$ of (\ref{E2-1R0})-(\ref{E2-2}) on $[0,T]\times M$. Moreover, for any $1\leq s\leq k-1$ and sufficient small $\varepsilon>0$, there holds
\begin{eqnarray}\label{E2-32}
|||h|||_{s,T}\leq c_7\left(|||h_0|||_{s,T}+|||h_1|||_{s,T}+|||g|||_{s,T}\right).
\end{eqnarray}
\end{theorem}
Before giving the proof of above theorem, we need to carry out some priori estimates on the solution $h$ of equation (\ref{E2-1R0}).
\begin{lemma}
Let $h$ be a $\textbf{H}^2$-solution to (\ref{E2-1R0}) with initial data (\ref{E2-2}). Assume that (\ref{E2-15})-(\ref{E2-15R1}) holds. Then
 there exists a nonnegative $C^1$ function $\varphi$ in $[0,T]\times M$ satisfying
\begin{eqnarray}\label{E2-3}
\frac{\partial_t\varphi}{\varphi}\leq c_8,~~\frac{|D\varphi|}{\varphi}\varrho\rho\leq c_9.
\end{eqnarray}
For any $\lambda>c_{10}$, there holds
\begin{eqnarray}\label{E2-4}
(\lambda-c_{10})\int_{0}^T\int_{M}e^{-\lambda t}\varphi(h_t^2+\varrho|Dh|^2+h^2)\sqrt{w}
&\leq&\int_{M}e^{-\lambda t}\varphi(h_1^2+\varrho|Dh_0|^2+h_0^2)\sqrt{w}\nonumber\\
&&+\int_{0}^T\int_{M}e^{-\lambda t}\varphi(\varepsilon f^2(\int_0^{t}\partial_ah)^2+g^2)\sqrt{w}.~~~~~~
\end{eqnarray}
\end{lemma}
\begin{proof}
Taking the inner product of the linearized equation (\ref{E2-1R0}) with $2e^{-\lambda t}\varphi h_t$, we have
\begin{eqnarray}\label{E2-5}
\partial_t[e^{-\lambda t}\varphi(h_t^2+\varrho\rho^{ab}\partial_ah\partial_bh)]&+&\lambda e^{-\lambda t}\varphi(h_t^2+\varrho\rho^{ab}\partial_ah\partial_bh)-2e^{-\lambda t}\partial_a(\varrho\rho^{ab}\varphi h_t\partial_bh)\nonumber\\
&=&e^{-\lambda t}\varphi_th_t^2+e^{-\lambda t}(\partial_t\varphi)\varrho\rho^{ab}\partial_ah\partial_bh
+e^{-\lambda t}\varphi\varrho(\partial_t\rho^{ab})\partial_ah\partial_bh\nonumber\\
&&-2e^{-\lambda t}(\partial_a\varphi)\varrho\rho^{ab}h_t\partial_bh-2e^{-\lambda t}\varphi(B-\varrho\partial_a\rho^{ab})\partial_bhh_t\nonumber\\
&&-2e^{-2\lambda t}(\partial_a\varrho)\rho^{ab}\varphi h_t\partial_bh+2\varepsilon e^{-\lambda t}\varphi h_tf\int_{0}^{t}\partial_ah+2e^{-\lambda t}\varphi h_tg.~~~~~~~
\end{eqnarray}
Here one should notice the Levi condition $(\ref{E2-15R1})$ to overcome the influence of the term $-2e^{-2\lambda t}(\partial_a\varrho)\rho^{ab}\varphi h_t\partial_bh$. One can see $(\ref{E2-7})$ and (\ref{E2-8}) for more details on this point.

Note that
\begin{eqnarray}\label{E2-6}
\partial_t(e^{-\lambda t}\varphi h^2)+\lambda e^{-\lambda t}\varphi h^2=\varphi_te^{-\lambda t}h^2+2e^{-\lambda t}\varphi hh_t.
\end{eqnarray}
Summing up (\ref{E2-5})-(\ref{E2-6}), we get
\begin{eqnarray}\label{E2-7}
\partial_t[e^{-\lambda t}\varphi(h_t^2+\varrho\rho^{ab}\partial_ah\partial_bh+h^2)]&+&\lambda e^{-\lambda t}\varphi(h_t^2+\varrho\rho^{ab}\partial_ah\partial_bh+h^2)-2e^{-\lambda t}\partial_a(\varrho\rho^{ab}\varphi h_t\partial_bh)\nonumber\\
&=&e^{-\lambda t}\varphi_t(h_t^2+h^2+2hh_t)+e^{-\lambda t}\varphi_t\varrho\rho^{ab}\partial_ah\partial_bh\nonumber\\
&&+e^{-\lambda t}\varphi\varrho(\partial_t\rho^{ab})\partial_ah\partial_bh-2e^{-\lambda t}\partial_a\varphi\varrho\rho^{ab}h_t\partial_bh\nonumber\\
&&-2e^{-\lambda t}\varphi(B-\partial_a(\varrho\rho^{ab}))\partial_bhh_t+2\varepsilon e^{-\lambda t}\varphi h_tf\int_{0}^{t}\partial_ah\nonumber\\
&&+2e^{-\lambda t}\varphi h_tg.~~~~~~
\end{eqnarray}
Using Cauchy inequality, by (\ref{E2-15}), (\ref{E2-15R1}), (\ref{E2-3}) and (\ref{E2-7}), we derive
\begin{eqnarray}\label{E2-8}
\partial_t[e^{-\lambda t}\varphi(h_t^2+\varrho\rho^{ab}\partial_ah\partial_bh+h^2)]&+&\lambda e^{-\lambda t}\varphi(h_t^2+\varrho\rho^{ab}\partial_ah\partial_bh+h^2)-2e^{-\lambda t}\partial_a(\varrho\rho^{ab}\varphi h_t\partial_bh)\nonumber\\
&\leq&(2\frac{\varphi_t}{\varphi}+\frac{|\partial_a\varphi|}{\varphi}\varrho\rho^{ab}+|B-\partial_a(\varrho\rho^{ab})|+2)e^{-\lambda t}\varphi h_t^2\nonumber\\
&&+(\frac{\varphi_t}{\varphi}\rho^{ab}+\partial_t\rho^{ab})\varrho e^{-\lambda t}\varphi\partial_ah\partial_bh\nonumber\\
&&+(\frac{|\partial_b\varphi|}{\varphi}\varrho\rho^{ab}+|B-\partial_a(\varrho\rho^{ab})|)e^{-\lambda t}\varphi(\partial_bh)^2\nonumber\\
&&+2\frac{\varphi_t}{\varphi}e^{-\lambda t}\varphi h^2+e^{-\lambda t}\varphi f^2(\int_0^{t}\partial_ah)^2+e^{-\lambda t}\varphi g^2\nonumber\\
&\leq&c_{10} e^{-\lambda t}\varphi(h_t^2+\varrho\rho^{ab}\partial_ah\partial_bh+h^2)\nonumber\\
&&+\varepsilon e^{-\lambda t}\varphi f^2(\int_0^{t}\partial_ah)^2+e^{-\lambda t}\varphi g^2.
\end{eqnarray}
Choosing $\lambda>c_{10}$, then inequality (\ref{E2-8}) leads to
\begin{eqnarray*}
\partial_t[e^{-\lambda t}\varphi(h_t^2+\varrho\rho^{ab}\partial_ah\partial_bh+h^2)]&+&(\lambda-c_{10})e^{-\lambda t}\varphi(h_t^2+\varrho\rho^{ab}\partial_ah\partial_bh+h^2)\nonumber\\
&&-2e^{-\lambda t}\partial_a(\varrho\rho^{ab}\varphi h_t\partial_bh)\nonumber\\
&\leq&\varepsilon e^{-\lambda t}\varphi f^2(\int_0^{t}\partial_ah)^2+e^{-\lambda t}\varphi g^2.
\end{eqnarray*}
Thus we obtain
\begin{eqnarray*}
\int_{M}e^{-\lambda t}\varphi(h_t^2&+&\varrho\rho^{ab}\partial_ah\partial_bh+h^2)(T)\sqrt{w}\nonumber\\
&+&(\lambda-c_{10})\int_0^T\int_Me^{-\lambda t}\varphi(h_t^2+\varrho\rho^{ab}\partial_ah\partial_bh+h^2)\sqrt{w}\nonumber\\
&\leq&\int_{M}e^{-\lambda t}\varphi(h_1^2+\varrho\rho^{ab}\partial_ah_0\partial_bh_0+h_0^2)\sqrt{w}\nonumber\\
&&+\int_0^T\int_Me^{-\lambda t}\varphi(\varepsilon f^2(\int_0^{t}\partial_ah)^2+g^2)\sqrt{w},
\end{eqnarray*}
which combining with (\ref{E2-15}) gives (\ref{E2-4}). This completes the proof.
\end{proof}
To derive classical energy estimates of solutions by eliminating the weight $\varphi$ in Lemma 1, we need to assume that
\begin{eqnarray}\label{E2-16}
0<\varrho(x)\leq 1,~~x\in M.
\end{eqnarray}
Then we have the following result by taking $\varphi(t,x)=\varrho(x)^{-1}$ in Lemma 1 for $\forall (t,x)\in[0,T]\times M$.
\begin{lemma}
Let $h$ be a $\textbf{H}^2$-solution to (\ref{E2-1R0}) with initial data (\ref{E2-2}). Assume that (\ref{E2-15}), (\ref{E2-15R1}) and (\ref{E2-16}) hold.
Then for any $\lambda>c_{11}$, there holds
\begin{eqnarray}\label{E2-17}
(\lambda-c_{11})\int_0^T\int_Me^{-\lambda t}(\varrho^{-1}h_t^2&+&|Dh|^2+\varrho^{-1}h^2)\sqrt{w}\nonumber\\
&\leq&\int_{M}e^{-\lambda t}(\varrho^{-1}h_1^2+|Dh_0|^2+\varrho^{-1}h_0^2)\sqrt{w}\nonumber\\
&&+\int_0^T\int_Me^{-\lambda t}\varrho^{-1}(\varepsilon f^2(\int_0^{t}\partial_ah)^2+g^2)\sqrt{w}.~~~~~
\end{eqnarray}
\end{lemma}
Next we eliminate $\varrho^{-1}(x)$ in energy inequality (\ref{E2-17}).
\begin{lemma}
Let $h$ be a $\textbf{H}^2$-solution to (\ref{E2-1R0}) with initial data (\ref{E2-2}).
Assume that (\ref{E2-15}), (\ref{E2-15R1}) and (\ref{E2-16}) hold.
 Then for any $\lambda>\lambda_0$, there holds
\begin{eqnarray}\label{E2-26}
\lambda\int_0^T\int_Me^{-\lambda t}(h_t^2+(Dh)^2+h^2)\sqrt{w}
&\leq&c_{12}\int_{M}e^{-\lambda t}(h_0^2+h_1^2+(Dh_0)^2+(Dh_1)^2)\sqrt{w}\nonumber\\
&&+c_{12}\int_0^T\int_Me^{-\lambda t}(|g|^2+|Dg|^2)\sqrt{w}\nonumber\\
&&+c_{12}\varepsilon\int_0^T\int_Me^{-\lambda t}f^2(\int_0^{t}\partial_ah)^2\sqrt{w},
\end{eqnarray}
where $\lambda_0$ depends on $\rho_0$, $\rho_1$ and the $\textbf{C}^1$ norm of $B$.
\end{lemma}
\begin{proof}
we introduce an auxiliary function $\tilde{h}$, which satisfies
\begin{eqnarray}\label{E2-18}
&&\partial_{tt}\tilde{h}-\partial_t\tilde{h}-\tilde{h}=g,~~x\in M,\\
&&\tilde{h}(0,x)=h_0,~~\partial_t\tilde{h}(0,x)=h_1.\nonumber
\end{eqnarray}

Let
\begin{eqnarray*}
\hat{h}=h-\tilde{h},
\end{eqnarray*}
then it follows from (\ref{E2-1R0}) and (\ref{E2-18}) that
\begin{eqnarray}\label{E2-1R2}
&&\partial_{tt}\hat{h}-\varrho(x)\partial_a(\rho^{ab}(t,x)\partial_b\hat{h})+B(t,x)\partial_a\hat{h}-\varepsilon f\int_0^{t}\partial_ah=\hat{g},~~x\in M,\\
&&\hat{h}(0,x)=0,~~\partial_t\hat{h}(0,x)=0,~~x\in M,
\end{eqnarray}
where
\begin{eqnarray}\label{E2-19}
\hat{g}=\varrho(x)\partial_a(\rho^{ab}(t,x)\partial_b\tilde{h})+\partial_t\tilde{h}+\tilde{h}-B(t,x)\partial_a\tilde{h}.
\end{eqnarray}
Using the similar deriving process with (\ref{E2-7}) and taking $\varphi=\varrho^{-1}$, by (\ref{E2-1R2}) we get
\begin{eqnarray}\label{E2-7R1}
\partial_t[e^{-\lambda t}(\varrho^{-1}\hat{h}_t^2+\rho^{ab}\partial_a\hat{h}\partial_b\hat{h}&+&\varrho^{-1}\hat{h}^2)]+\lambda e^{-\lambda t}(\varrho^{-1}\hat{h}_t^2+\rho^{ab}\partial_a\hat{h}\partial_b\hat{h}+\varrho^{-1}\hat{h}^2)\nonumber\\
&&-2e^{-\lambda t}\partial_a(\rho^{ab}\hat{h}_t\partial_b\hat{h})\nonumber\\
&=&-2e^{-\lambda t}\partial_a\varrho^{-1}\varrho\rho^{ab}\hat{h}_t\partial_b\hat{h}
-2e^{-\lambda t}(-\varrho^{-1}B+\partial_a\rho^{ab})\partial_b\hat{h}\hat{h}_t\nonumber\\
&&+e^{-\lambda t}(\partial_t\rho^{ab})\partial_ah\partial_bh-2\varepsilon e^{-\lambda t}\varrho^{-1}\hat{h}_tf\int_{0}^{t}\partial_ah+2e^{-\lambda t}\varrho^{-1}\hat{h}_t\hat{g}.~~~~~~
\end{eqnarray}
To avoid an extra loss of derivatives of integrating (\ref{E2-7R1}), by (\ref{E2-19}), we can rewrite the last term in (\ref{E2-7R1}) as
\begin{eqnarray}\label{E2-20}
e^{-\lambda t}\varrho^{-1}\hat{h}_t\hat{g}&=&\partial_a(e^{-\lambda t}\rho^{ab}\partial_b\tilde{h}\partial_t\hat{h})-\partial_t(e^{-\lambda t}\rho^{ab}\partial_b\tilde{h}\partial_a\hat{h})+e^{-\lambda t}\rho^{ab}\partial_a\hat{h}\partial_b\tilde{h}_t-\lambda e^{-\lambda t}\rho^{ab}\partial_a\hat{h}\partial_b\tilde{h}\nonumber\\
&&+e^{-\lambda t}(\partial_t\rho^{ab})\partial_a\tilde{h}\partial_a\hat{h}+e^{-\lambda t}\varrho^{-1}(\hat{h}_t\tilde{h}_t+\hat{h}_t\tilde{h})
-e^{-\lambda t}\varrho^{-1}\hat{h}_tB\partial_a\tilde{h}.~~~~~~
\end{eqnarray}
Inserting (\ref{E2-20}) into (\ref{E2-7R1}), using Cauchy inequality, (\ref{E2-15}), (\ref{E2-15R1}) and (\ref{E2-16}), then integrating in $[0,T]\times M$, we obtain
\begin{eqnarray}\label{E2-21}
\int_{M}e^{-\lambda t}(\varrho^{-1}\hat{h}_t^2&+&(D\hat{h})^2+\varrho^{-1}\hat{h}^2)(T)\sqrt{w}\nonumber\\
&+&(\lambda-c_{13})\int_0^T\int_Me^{-\lambda t}(\varrho^{-1}\hat{h}_t^2+(D\hat{h})^2+\varrho^{-1}\hat{h}^2)\sqrt{w}\nonumber\\
&\leq&\int_{M}e^{-\lambda t}(D\tilde{h})(T)^2\sqrt{w}\nonumber\\
&&+c_{14}\int_0^T\int_Me^{-\lambda t}((\partial_b\tilde{h}_t)^2+(D\tilde{h})^2+\tilde{h}_t^2+\tilde{h}^2)\sqrt{w}\nonumber\\
&&+c_{14}\varepsilon\int_0^T\int_Me^{-\lambda t}f^2(\int_0^{t}\partial_ah)^2\sqrt{w}.
\end{eqnarray}
It follows from (\ref{E2-16}) that
\begin{eqnarray}\label{E2-22}
\varrho^{-1}\hat{h}_t^2+(D\hat{h})^2+\varrho^{-1}\hat{h}^2\geq|h_t-\tilde{h}_t|^2+|Dh-D\tilde{h}|^2+|h-\tilde{h}|^2,
\end{eqnarray}
which combing with (\ref{E2-21}) gives that
\begin{eqnarray}\label{E2-23}
c_{15}\int_0^T\int_Me^{-\lambda t}(h_t^2+(Dh)^2+h^2)\sqrt{w}
&\leq&\int_{M}e^{-\lambda t}(\tilde{h}_t^2+(D\tilde{h})^2+\tilde{h}^2)(T)\sqrt{w}\nonumber\\
&&+c_{16}\int_0^T\int_Me^{-\lambda t}(\tilde{h}_t^2+(\partial_b\tilde{h}_t)^2+(D\tilde{h})^2+\tilde{h}^2)\sqrt{w}\nonumber\\
&&+c_{16}\varepsilon\int_0^T\int_Me^{-\lambda t}f^2(\int_0^{t}\partial_ah)^2\sqrt{w}.~~~~~
\end{eqnarray}
Multiplying (\ref{E2-18}) both side by $2e^{-\lambda t}\tilde{h}$ and integrating on $[0,T]\times M$, we have
\begin{eqnarray}\label{E2-24}
\int_{M}e^{-\lambda t}(\tilde{h}^2+(\partial_t\tilde{h})^2)(T)\sqrt{w}&+&\lambda\int_0^T\int_Me^{-\lambda t}(\tilde{h}^2+(\partial_t\tilde{h})^2)\sqrt{w}\nonumber\\
&\leq& c_{17}\int_{M}e^{-\lambda t}(h_0^2+h_1^2)\sqrt{w}\nonumber\\
&&+c_{17}\int_0^T\int_Me^{-\lambda t}|g|^2\sqrt{w}.~~~~~
\end{eqnarray}
Differentiating (\ref{E2-18}) with respect to $x$, by the similar process of getting (\ref{E2-24}), we derive
\begin{eqnarray}\label{E2-25}
\int_{M}e^{-\lambda t}((D\tilde{h})^2+(\partial_tD\tilde{h})^2)(T)\sqrt{w}&+&\lambda\int_0^T\int_Me^{-\lambda t}((D\tilde{h})^2+(\partial_tD\tilde{h})^2)\sqrt{w}\nonumber\\
&\leq& c_{18}\int_{M}e^{-\lambda t}(h_0^2+h_1^2+(Dh_0)^2+(Dh_1)^2)\sqrt{w}\nonumber\\
&&+c_{18}\int_0^T\int_Me^{-\lambda t}(|g|^2+|Dg|^2)\sqrt{w}.
\end{eqnarray}
Inequalities (\ref{E2-24})-(\ref{E2-25}) give the control of two terms in the right hand side of (\ref{E2-23}). Thus substituting (\ref{E2-24})-(\ref{E2-25}) into (\ref{E2-23}), we obtain (\ref{E2-26}). This completes the proof.
\end{proof}
In what follows, we plan to obtain the estimate for $|||h|||_{s,T}$ by considering the equations of time-space derivatives of $h$. For any multi-index $\alpha\in\textbf{Z}^2_+$ with $|\alpha|=s$, applying $D^{\alpha}$ to both sides of (\ref{E2-1R0}) to get
\begin{eqnarray}\label{E2-27}
\partial_{tt}D^{\alpha}h-\varrho\partial_a(\rho^{ab}\partial_bD^{\alpha}h)+B\partial_aD^{\alpha}h-\varepsilon f(Du,Dv)\int_0^{t}D^{\alpha+1}h=F_{\alpha},
\end{eqnarray}
where the nonlinear terms
\begin{eqnarray}\label{E2-27R0}
F_{\alpha}=D^{\alpha}g&+&\sum_{s_1+s_2=\alpha,~|s_1|\geq1}(D^{s_1}(\varrho\rho^{ab}))\partial_{a}\partial_bD^{s_2}h\nonumber\\
&&+\sum_{s_1+s_2=\alpha,~|s_1|\geq1}(D^{s_1}(\varrho\partial_a\rho^{ab}))\partial_bD^{s_2}h-\sum_{s_1+s_2=\alpha,~|s_1|\geq1}(D^{s_1}B)\partial_aD^{s_2}h\nonumber\\
&&+\varepsilon\sum_{s_1+s_2=\alpha,~|s_1|\geq1}(D^{s_1}f(Du,Dv))\int_0^tDD^{s_2}h.~~~~~~
\end{eqnarray}
For convenience, we denote all spacial derivatives of $h$ of the order $s$ by a column vector of $m(s)$ components
\begin{eqnarray*}
u_s^T=(\partial_1^sh,\partial_1^{s-1}\partial_2h,\ldots,\partial_2^sh).
\end{eqnarray*}
It follows from putting together the equations corresponding to all $\alpha$ with $|\alpha|=s$ in (\ref{E2-27}) that
\begin{eqnarray}\label{E2-27R1}
\partial_{tt}u_s-\varrho\partial_a(\rho^{ab}\partial_bu_s)+B_s\partial_au_s-\varepsilon f_s(Du,Dv)\int_0^{t}Du_s=\tilde{F}_{\alpha},
\end{eqnarray}
where $B_s$, $f_s$ are $(m\times m)$-matrices and $\tilde{F}_{\alpha}$ is an $m$-vector given by
\begin{eqnarray*}
\tilde{F}_{\alpha}^T=(F_{(s,0,\ldots,0)},F_{(s-1,1,\ldots,0)},\ldots,F_{(0,0,\ldots,s)}).
\end{eqnarray*}
It is obviously that the following Levi condition holds
\begin{eqnarray*}
|B_s-\partial_a(\varrho\rho_{ab})\xi_a\xi_b|\leq c_{19}\varrho,~~x\in M.
\end{eqnarray*}
For $s\geq1$, we set
\begin{eqnarray*}
C_s=\sum_{a,b=1}^2(\|D^2(\varrho\rho^{ab})\|_s+\|D(\varrho\partial_a\rho^{ab}I)\|_s)+\|B_s\|_s.
\end{eqnarray*}

\begin{lemma}
Let $h$ be a $\textbf{H}^2$-solution to (\ref{E2-1R0}) with initial data (\ref{E2-2}).
Assume that (\ref{E2-15}), (\ref{E2-15R1}) and (\ref{E2-16}) holds.
 Then for sufficient small $\varepsilon>0$, there holds
\begin{eqnarray}\label{E2-28}
|||h|||_{s,T}\leq c_{20}\left(||h_0||_{s,T}+||h_1||_{s,T}+|||g|||_{s,T}\right).
\end{eqnarray}
\end{lemma}
\begin{proof}
The proof is based on the induction. For $s=1$, Lemma 3 gives the result by choosing a suitable $\lambda$. Assume that (\ref{E2-28}) holds for all $1\leq j\leq s$, then we prove that (\ref{E2-28}) holds for $s+1$. Since (\ref{E2-27R1}) has the same structure with (\ref{E2-1R0}), the result in Lemma 3 can be used directly.
Note that the vector $D^sh$ satisfies (\ref{E2-27R1}). By (\ref{E2-26}), we have
\begin{eqnarray}\label{E2-29}
&&\lambda\int_0^T\int_Me^{-\lambda t}((D^s\partial_th)^2+(D^{s+1}h)^2+(D^sh)^2)\sqrt{w}\nonumber\\
&\leq&c_{21}\int_{M}e^{-\lambda t}((D^sh_0)^2+(D^sh_1)^2+(D^{s+1}h_0)^2+(D^{s+1}h_1)^2)\sqrt{w}\nonumber\\
&&+c_{21}\sum_{|\alpha|=s}\int_0^T\int_Me^{-\lambda t}(|F_{\alpha}|^2+|DF_{\alpha}|^2)\sqrt{w}\nonumber\\
&&+c_{21}\varepsilon\int_0^T\int_Me^{-\lambda t}f^2(\int_0^{t}D^{s+1}h)^2\sqrt{w},~~~~~~
\end{eqnarray}
where $c$ is a constant depending on $C_s$.

By (\ref{E2-27R0}), we drive
\begin{eqnarray}\label{E2-29R0}
|F_{\alpha}|^2&+&|DF_{\alpha}|^2\leq|D^{s}g|^2+|D^{s+1}g|^2\nonumber\\
&&+c_{22}(|D^{s+1}h|^2+\varepsilon|Df||\int_0^tD^{s+1}h|)+|F'_{\alpha}|^2,~~~
\end{eqnarray}
where $F'_{\alpha}$ denotes the nonlinear terms involving $D^{\alpha}$ with $|\alpha|\leq s$.

It follows from (\ref{E2-29}) and (\ref{E2-29R0}) that
\begin{eqnarray}\label{E2-29R2}
&&\int_0^T\int_Me^{-\lambda t}(\lambda(D^s\partial_th)^2+(\lambda-c)(D^{s+1}h)^2+\lambda(D^sh)^2)\sqrt{w}\nonumber\\
&\leq&c_{23}\int_{M}e^{-\lambda t}((D^sh_0)^2+(D^sh_1)^2+(D^{s+1}h_0)^2+(D^{s+1}h_1)^2)\sqrt{w}\nonumber\\
&&+c_{23}\int_0^T\int_Me^{-\lambda t}(|D^{s}g|^2+|D^{s+1}g|^2+c\varepsilon|Df||\int_0^tD^{s+1}h|+|F'_{\alpha}|^2)\sqrt{w}\nonumber\\
&&+c_{23}\varepsilon\int_0^T\int_Me^{-\lambda t}f^2\left(\int_0^{t}D^{s+1}h\right)^2\sqrt{w}.~~~~~~
\end{eqnarray}
Note that
\begin{eqnarray}\label{E2-29R3}
\|F'_{\alpha}\|_{\textbf{L}^2}\leq c_{24}(\|h\|_{\textbf{H}^s}+T^2\varepsilon\|f\|_{\textbf{H}^s}|||h|||_{s,T}).
\end{eqnarray}
Thus choosing a $\lambda$ large enough, by (\ref{E2-29R2})-(\ref{E2-29R3}), we get
\begin{eqnarray}\label{E2-30}
\|D^s\partial_th\|_{\textbf{L}^2}+\|D^sh\|_{\textbf{H}^1}&\leq&c_{25}(\|h_0\|_{\textbf{H}^{s+1}}+\|h_1\|_{\textbf{H}^{s+1}}+\|g\|_{\textbf{H}^{s+1}}+\|h\|_{\textbf{H}^{s}})\nonumber\\
&&+\varepsilon T^2\|f\|_{s+1}|||h|||_{s+1,T}.~~~~~~
\end{eqnarray}
Furthermore, we can apply $D^{s-1}$ to both sides of (\ref{E2-1R0}), then deriving a similar estimate with (\ref{E2-30}). We conclude that
\begin{eqnarray}\label{E2-31}
\sum_{j=0}^{s+1}\|\partial_t^jD^{s+1-j}h\|_{\textbf{L}^2}&\leq&c_{26}(\|h_0\|_{\textbf{H}^{s+1}}+\|h_1\|_{\textbf{H}^{s+1}}+\|g\|_{\textbf{H}^{s+1}}+\|h\|_{\textbf{H}^{s}})\nonumber\\
&&+\varepsilon T^2\|f\|_{\textbf{H}^{s+1}}|||h|||_{s+1,T}.~~~~~~
\end{eqnarray}
Note that $\|h\|_l$ is equivalent to $\|h\|_{\textbf{H}^l}+\|\partial\partial_th\|_{\textbf{H}^l}$. Hence (\ref{E2-28}) can be derived by (\ref{E2-31}).
This completes the proof.
\end{proof}

\textbf{Proof of Theorem 3.}
The proof is based on an approximation, which follows the idea of proof of Theorem 1 in Page 576 of \cite{O} or
Theorem 5.1 in \cite{Han0} (also see \cite{Han2,N}).
For any small $\delta>0$, we consider the regularized equation with the initial data (\ref{E2-2}) in $[0,T]\times M$
\begin{eqnarray*}
\partial_{tt}h-(\varrho(x)+\delta)\partial_a(\rho^{ab}(t,x)\partial_bh)+B(t,x)\partial_ah-\varepsilon f(Du,Dv)\int_0^{t}\partial_ah=g(t,x).~~~~
\end{eqnarray*}
Above system is a strictly hyperbolic equation, which is equivalent to
\begin{eqnarray}\label{E2-33}
\partial_t\textbf{h}+\textbf{A}\textbf{h}=\textbf{g},
\end{eqnarray}
with initial data $\textbf{h}=(h_0,h_1,0)^T$, where $\textbf{h}=(\tilde{h},h,z)^T$, $\textbf{g}=(g,0,0)^T$ and
\begin{eqnarray*}
\textbf{A}= \left(
\begin{array}{ccc}
 0& -(\varrho(x)+\delta)\partial_a(\rho^{ab}(t,x)\partial_b)+B(t,x)\partial_a & \varepsilon f  \\
 -1 & 0 & 0\\
 0 & -\partial_a & 0
\end{array}
\right).
\end{eqnarray*}
Note that $f$ is bounded in the norm $|||\cdot|||_{s,T}$ with $1\leq s\leq k-1$ and all coefficients in the operator $A$ are $\textbf{C}^k$ for some integer $k\geq2$. Then linear equation (\ref{E2-33}) admits an $\textbf{H}^k$-solution $\textbf{h}_{\delta}$. By (\ref{E2-28}) in Lemma 4, for $s\leq k-1$, we have
\begin{eqnarray}\label{E2-34}
|||h_{\delta}|||_{s,T}\leq c_{s}\left(||h_0||_{s,T}+||h_1||_{s,T}+|||g|||_{s,T}\right),~~~~
\end{eqnarray}
where $c_{s}$ denotes a constant which is \textbf{independent} of $\delta$. This is because the energy estimate of diffusion term to linearized equation (\ref{E2-34}) is independent of the term $(\varrho(x)+\delta)$. It is the same idea as the proof of Theorem 5.1 in page 456 of \cite{Han0} or Theorem 1.1 in page 332 of \cite{N}. Thus energy estimate (\ref{E2-34}) holds \textbf{uniformly} in $\delta$.
Then there exists a function $h$ and a sequence $\delta_j\longrightarrow0$ such that $h_{\delta_j}\longrightarrow h$ in $\textbf{H}^{k-2}$. Here $h$ is solution of (\ref{E2-1R0}). This completes the proof.

\begin{remark}
In fact, Theorem 3 gives a \textbf{local existence} result on a  linear variable coefficent wave equation (\ref{E2-1R0}) with the initial data (\ref{E2-2}). The constant $c_s$ in (\ref{E2-34}) is independent of $\delta$. Here we want to point out that blow up phenomenon of nonlinear wave equation mainly depends on the nonlinear term. For example, if the quasilinear term satisfies null condition or weak null condition, then the global existence of smooth solution of quasilinear wave equation can be proven.
\end{remark}

\textbf{Proof of Theorem 2}
By (\ref{E2-36}), directly computation shows that
\begin{eqnarray}\label{E2-40}
\partial_a(\epsilon^{ac}\epsilon^{bd}\gamma(u_0)_{cd}\partial_bh^m)&=&\gamma_0(x)\partial_a(\epsilon^{ac}\epsilon^{bd}\gamma(x)_{cd}\partial_bh^m)\nonumber\\
&&+(\partial_a\gamma_0(x))\epsilon^{ac}\epsilon^{bd}\gamma(x)_{cd}\partial_bh^m.~~~~
\end{eqnarray}
Using the second condition in (\ref{E2-38}) and (\ref{E2-37}), we derive
\begin{eqnarray*}
&&\|(\partial_a\gamma_0(x))\epsilon^{ac}\epsilon^{bd}\gamma(x)_{cd}\partial_bh^m\|_s\leq c_{\gamma_1}\|h^m\|_{s+1},\\
&&|||\epsilon^{ac}\epsilon^{bd}\partial_cu_0^m(\int_0^t\partial_dv_n)|||_{s,T}\leq c_{27}T\|u_0\|_{s+1}|||v|||_{s+1,T}\leq c_{T,R},\\
&&|||\epsilon^{ac}\epsilon^{bd}\partial_du_{0m}(\int_0^t\partial_cv^n)|||_{s,T}\leq c_{T,R},\\
&&|||\epsilon^{ac}\epsilon^{bd}\gamma(\int_0^tv)_{cd}|||_{s,T}\leq c_{T,R},
\end{eqnarray*}
where constants $c_{\gamma_1}$ and $c_{T,R}$ depend on $\gamma_1$ and $T$, $R$, respectively.

Thus we can set
\begin{eqnarray}\label{E2-14}
\varrho(x)(\partial_a\rho^{ab}(t,x)\partial_bh^m)-\partial_a(O(\varepsilon)\partial_bh^m)&=&\partial_a(\epsilon^{ac}\epsilon^{bd}\gamma(u_0)_{cd}\partial_bh^m)\nonumber\\
&&+\varepsilon\partial_a(\epsilon^{ac}\epsilon^{bd}\partial_cu_0^m(\int_0^t\partial_dv_n)\partial_bh^m)\nonumber\\
&&+\varepsilon\partial_a(\epsilon^{ac}\epsilon^{bd}\partial_du_{0m}(\int_0^t\partial_cv^n)\partial_bh^m)\nonumber\\
&&+\varepsilon\partial_a(\epsilon^{ac}\epsilon^{bd}\gamma(\int_0^tv)_{cd}\partial_bh^m).~~~~~
\end{eqnarray}
In a similar way with (\ref{E2-14}), we set
\begin{eqnarray*}
\label{E2-14R1}
(B(t,x)+O(\varepsilon))\partial_ah^m&=&-(\partial_a\gamma_0(x))\epsilon^{ac}\epsilon^{bd}\gamma(x)_{cd}\partial_bh^m+O(\varepsilon)\partial_ah^m\nonumber\\
&=&-(\partial_a\gamma_0(x))\epsilon^{ac}\epsilon^{bd}\gamma(x)_{cd}\partial_bh^m\nonumber\\
&&+2\varepsilon\partial_a(\epsilon^{ab}\epsilon^{cd}w^{-1}(\int_0^t\partial_cv^m)(\int_0^t\partial_dv^n)\partial_bh_n)\nonumber\\
&&-\varepsilon\epsilon^{ab}w^{-\frac{1}{2}}\partial_c(\epsilon^{cd}w^{-\frac{1}{2}})\gamma_{bd}\partial_ah^m\nonumber\\
&&-2\varepsilon\epsilon^{ab}w^{-\frac{1}{2}}\partial_c(\epsilon^{cd}w^{-\frac{1}{2}})(\int_0^t\partial_bv^n)(\int_0^t\partial_av^m)\partial_dh_n,~~~
\end{eqnarray*}
\begin{eqnarray*}\label{E2-14R2}
f(Du,Dv)\int_0^{t}\partial_ah^m&=&O(\varepsilon)\int_0^{t}\partial_ah^m=\varepsilon\partial_a(\epsilon^{ac}\epsilon^{bd}\partial_cu_0^m(\int_0^t\partial_dh_n)\partial_bv^m)\nonumber\\
&&+\varepsilon\partial_a(\epsilon^{ac}\epsilon^{bd}\partial_du_{0m}(\int_0^t\partial_ch^n)\partial_bv^m)\nonumber\\
&&+\varepsilon\partial_a(\epsilon^{ac}\epsilon^{bd}(\int_0^t\partial_ch^n)(\int_0^t\partial_dv_n)\partial_bv^m)\nonumber\\
&&+\varepsilon\partial_a(\epsilon^{ac}\epsilon^{bd}(\int_0^t\partial_cv^n)(\int_0^t\partial_dh_n)\partial_bv^m)\nonumber\\
&&-2\varepsilon\partial_a(\epsilon^{ab}\epsilon^{cd}w^{-1}(\int_0^t\partial_ch^m)(\int_0^t\partial_dv^n)\partial_bv_n)\nonumber\\
&&-2\varepsilon\partial_a(\epsilon^{ab}\epsilon^{cd}w^{-1}(\int_0^t\partial_cv^m)(\int_0^t\partial_dh^n)\partial_bv_n)\nonumber\\
&&+\varepsilon\epsilon^{ab}w^{-\frac{1}{2}}\partial_c(\epsilon^{cd}w^{-\frac{1}{2}})((\int_0^t\partial_bh^n)(\int_0^t\partial_dv_n)\partial_av^m+(\int_0^t\partial_bv^n)(\int_0^t\partial_dh_n)\partial_av^m)\nonumber\\
&&+2\varepsilon\epsilon^{ab}w^{-\frac{1}{2}}\partial_c(\epsilon^{cd}w^{-\frac{1}{2}})((\int_0^t\partial_bv^n)(\int_0^t\partial_ah^m)\partial_dv_n+(\int_0^t\partial_bh^n)(\int_0^t\partial_av^m)\partial_dv_n).
\end{eqnarray*}
Note that assumptions (\ref{E2-38})-(\ref{E2-37R}) are equivalent to (\ref{E2-15})-(\ref{E2-15R1}).
Therefore using the similar process of proof of Theorem 3, one can prove this result.

\section{Long time existence for the bosonic membrane equation}
In this section, we construct a solution $u\in\textbf{C}_T^{k}\cap\textbf{B}_{R,T}^k$ ($k\geq2$) of system (\ref{E1-1}) on $[0,\frac{T}{\sqrt{\varepsilon}}]\times M$ by a suitable Nash-Moser iteration scheme, where $\textbf{B}_{R,T}^k$ is defined in (\ref{E2-35}). We know that system (\ref{E1-1}) is equivalent to system (\ref{E2-1R1}) by a arrangement in section 2. Thus the main goal is to solve system (\ref{E2-1R1}).

Rescaling in (\ref{E2-1R1}) amplitude and time as
\begin{eqnarray*}
v^m(t,x)\mapsto\varepsilon^2 v^m(\sqrt{\varepsilon}t,x),~~\varepsilon>0,
\end{eqnarray*}
we are to prove the existence solution on $[0,T]\times M$ of
\begin{eqnarray}\label{E3-1}
\partial_{tt}v^m-\varepsilon^{-1}\partial_a(\epsilon^{ac}\epsilon^{bd}\gamma(u_0)_{cd}\partial_bv^m)=\varepsilon^2\mathcal{F}(v^m),
\end{eqnarray}
with initial data
\begin{eqnarray}\label{E3-1R0}
v^m(0)=v^m_0,~~\partial_tv^m(0)=v^m_1,
\end{eqnarray}
where
\begin{eqnarray}\label{E3-2}
\mathcal{F}(v^m)&=&\partial_a(\epsilon^{ac}\epsilon^{bd}\partial_cu_0^m(\int_0^t\partial_dv_n)\partial_bv^m)+\partial_a(\epsilon^{ac}\epsilon^{bd}\partial_du_{0m}(\int_0^t\partial_cv^n)\partial_bv^m)\nonumber\\
&&+\varepsilon\partial_a(\epsilon^{ac}\epsilon^{bd}\gamma(\int_0^tv)_{cd}\partial_bv^m)
-2\varepsilon\partial_a\left(\epsilon^{ab}\epsilon^{cd}w^{-1}(\int_0^t\partial_cv^m)(\int_0^t\partial_dv^n)\partial_bv_n\right)\nonumber\\
&&+\varepsilon\epsilon^{ab}w^{-\frac{1}{2}}\partial_c(\epsilon^{cd}w^{-\frac{1}{2}})\left(\gamma_{bd}\partial_av^m+2(\int_0^t\partial_av^m)(\int_0^t\partial_bv^n)\partial_dv_n\right).~~~~~
\end{eqnarray}
Introduce an auxiliary function
\begin{eqnarray}\label{E3-1R2}
W^m(t,x)=v^m(t,x)-v^m_0-v^m_1t,
\end{eqnarray}
then the initial value problem (\ref{E3-1})-(\ref{E3-1R0}) is equivalent to
\begin{eqnarray}\label{E3-1R1}
\mathcal{G}(W^m):=\partial_{tt}W^m-\varepsilon^{-1}\partial_a(\epsilon^{ac}\epsilon^{bd}\gamma(u_0)_{cd}\partial_bW^m)-\varepsilon^2\mathcal{F}(W^m)=0,
\end{eqnarray}
with zero initial data
\begin{eqnarray*}
W^m(0)=0,~~\partial_tW^m(0)=0.
\end{eqnarray*}
We treat problem (\ref{E3-1R1}) iteratively as a small perturbation of the linear degenerate hyperbolic equation (\ref{E2-13}).
Linearizing nonlinear system (\ref{E3-1R1}), we obtain the linearized operator
\begin{eqnarray}\label{E3-3}
\mathcal{L}_{\varepsilon}h^m=\mathcal{L}h^m-\varepsilon^2\partial_W\mathcal{F}(W^m)h^m,
\end{eqnarray}
where
\begin{eqnarray*}
\mathcal{L}h^m=\partial_{tt}h^m-\varepsilon^{-1}\partial_a(\epsilon^{ac}\epsilon^{bd}\gamma(u_0)_{cd}\partial_bh^m),\\
\end{eqnarray*}
\begin{eqnarray}\label{E3-2R0}
\partial_W\mathcal{F}(W^m)h^m&=&\partial_a(\epsilon^{ac}\epsilon^{bd}\partial_cu_0^m(\int_0^t\partial_dW_n)\partial_bh^m)
+\partial_a(\epsilon^{ac}\epsilon^{bd}\partial_cu_0^m(\int_0^t\partial_dh_n)\partial_bW^m)\nonumber\\
&&+\partial_a(\epsilon^{ac}\epsilon^{bd}\partial_du_{0m}(\int_0^t\partial_cW^n)\partial_bh^m)+\partial_a(\epsilon^{ac}\epsilon^{bd}\partial_du_{0m}(\int_0^t\partial_ch^n)\partial_bW^m)\nonumber\\
&&+\varepsilon\partial_a(\epsilon^{ac}\epsilon^{bd}\gamma(\int_0^tW)_{cd}\partial_bh^m)+\varepsilon\partial_a(\epsilon^{ac}\epsilon^{bd}(\int_0^t\partial_ch^n)(\int_0^t\partial_dW_n)\partial_bW^m)\nonumber\\
&&+\varepsilon\partial_a(\epsilon^{ac}\epsilon^{bd}(\int_0^t\partial_cW^n)(\int_0^t\partial_dh_n)\partial_bW^m)\nonumber\\
&&-2\varepsilon\partial_a(\epsilon^{ab}\epsilon^{cd}w^{-1}(\int_0^t\partial_ch^m)(\int_0^t\partial_dW^n)\partial_bW_n)\nonumber\\
&&-2\varepsilon\partial_a(\epsilon^{ab}\epsilon^{cd}w^{-1}(\int_0^t\partial_cW^m)(\int_0^t\partial_dh^n)\partial_bW_n)\nonumber\\
&&-2\varepsilon\partial_a(\epsilon^{ab}\epsilon^{cd}w^{-1}(\int_0^t\partial_cW^m)(\int_0^t\partial_dW^n)\partial_bh_n)\nonumber\\
&&+\varepsilon\epsilon^{ab}w^{-\frac{1}{2}}\partial_c(\epsilon^{cd}w^{-\frac{1}{2}})(\gamma_{bd}\partial_ah^m+(\int_0^t\partial_bh^n)(\int_0^t\partial_dW_n)\partial_aW^m)\nonumber\\
&&+\varepsilon\epsilon^{ab}w^{-\frac{1}{2}}\partial_c(\epsilon^{cd}w^{-\frac{1}{2}})((\int_0^t\partial_bW^n)(\int_0^t\partial_dh_n)\partial_aW^m\nonumber\\
&&+2\varepsilon(\int_0^t\partial_bW^n)(\int_0^t\partial_ah^m)\partial_dW_n)\nonumber\\
&&+2\varepsilon\epsilon^{ab}w^{-\frac{1}{2}}\partial_c(\epsilon^{cd}w^{-\frac{1}{2}})((\int_0^t\partial_bh^n)(\int_0^t\partial_aW^m)\partial_dW_n\nonumber\\
&&+(\int_0^t\partial_bW^n)(\int_0^t\partial_aW^m)\partial_dh_n).
\end{eqnarray}
For the nonlinear term, by (\ref{E3-2}) and (\ref{E3-2R0}), direct computations show that
\begin{eqnarray}\label{E3-4}
R(h^m)&:=&\mathcal{F}(W^m+h^m)-\mathcal{F}(W^m)-\partial_W\mathcal{F}(W^m)h^m\nonumber\\
&=&\partial_a(\epsilon^{ac}\epsilon^{bd}\partial_cu_0^m(\int_0^t\partial_dh_n)\partial_bh^m)+\partial_a(\epsilon^{ac}\epsilon^{bd}\partial_du_{0m}(\int_0^t\partial_ch^n)\partial_bh^m)\nonumber\\
&&+\varepsilon\partial_a(\epsilon^{ac}\epsilon^{bd}\gamma(\int_0^th)_{cd}\partial_bh^m)+\varepsilon\partial_a(\epsilon^{ac}\epsilon^{bd}(\int_0^t\partial_cW_n)(\int_0^t\partial_dh^n)\partial_bh^m)\nonumber\\
&&+\varepsilon\partial_a(\epsilon^{ac}\epsilon^{bd}\gamma(\int_0^th)_{cd}\partial_bW^m)+\varepsilon\partial_a(\epsilon^{ac}\epsilon^{bd}(\int_0^t\partial_ch_n)(\int_0^t\partial_dW^n)\partial_bh^m)\nonumber\\
&&-2\varepsilon\partial_a\left(\epsilon^{ab}\epsilon^{cd}w^{-1}(\int_0^t\partial_ch^m)(\int_0^t\partial_dh^n)\partial_bh_n\right)\nonumber\\
&&-2\varepsilon\partial_a\left(\epsilon^{ab}\epsilon^{cd}w^{-1}(\int_0^t\partial_ch^m)(\int_0^t\partial_dh^n)\partial_bW_n\right)\nonumber\\
&&-2\varepsilon\partial_a\left(\epsilon^{ab}\epsilon^{cd}w^{-1}(\int_0^t\partial_cW^m)(\int_0^t\partial_dh^n)\partial_bh_n\right)\nonumber\\
&&-2\varepsilon\partial_a\left(\epsilon^{ab}\epsilon^{cd}w^{-1}(\int_0^t\partial_ch^m)(\int_0^t\partial_dW^n)\partial_bh_n\right)\nonumber\\
&&+\varepsilon\epsilon^{ab}w^{-\frac{1}{2}}\partial_c(\epsilon^{cd}w^{-\frac{1}{2}})(\gamma(h)_{bd}\partial_ah^m)\nonumber\\
&&+\varepsilon\epsilon^{ab}w^{-\frac{1}{2}}\partial_c(\epsilon^{cd}w^{-\frac{1}{2}})(\gamma(h)_{bd}\partial_aW^m)\nonumber\\
&&+\varepsilon\epsilon^{ab}w^{-\frac{1}{2}}\partial_c(\epsilon^{cd}w^{-\frac{1}{2}})((\int_0^t\partial_bW^n)(\int_0^t\partial_dh_n)\partial_ah^m)\nonumber\\
&&+\varepsilon\epsilon^{ab}w^{-\frac{1}{2}}\partial_c(\epsilon^{cd}w^{-\frac{1}{2}})((\int_0^t\partial_bh^n)(\int_0^t\partial_dW_n)\partial_ah^m)\nonumber\\
&&+2\varepsilon\epsilon^{ab}w^{-\frac{1}{2}}\partial_c(\epsilon^{cd}w^{-\frac{1}{2}})(\int_0^t\partial_ah^m)(\int_0^t\partial_bh^n)\partial_dh_n\nonumber\\
&&+2\varepsilon\epsilon^{ab}w^{-\frac{1}{2}}\partial_c(\epsilon^{cd}w^{-\frac{1}{2}})(\int_0^t\partial_ah^m)(\int_0^t\partial_bh^n)\partial_dW_n\nonumber\\
&&+2\varepsilon\epsilon^{ab}w^{-\frac{1}{2}}\partial_c(\epsilon^{cd}w^{-\frac{1}{2}})(\int_0^t\partial_ah^m)(\int_0^t\partial_bW^n)\partial_dh_n\nonumber\\
&&+2\varepsilon\epsilon^{ab}w^{-\frac{1}{2}}\partial_c(\epsilon^{cd}w^{-\frac{1}{2}})(\int_0^t\partial_aW^m)(\int_0^t\partial_bh^n)\partial_dh_n.
\end{eqnarray}
Using the product estimate $\|uv\|_s\leq\|u\|_s\|v\|_s$ (see (2.18) in \cite{Allen1}), directly estimating each terms in (\ref{E3-4}), we obtain the following result.
\begin{lemma}
For any $s\geq2$, there holds
\begin{eqnarray}\label{E3-5}
|||R(h)|||_{s,T}\leq c_{28}\left(|||h|||_{s+2,T}^2(1+|||W|||_{s+1,T})+|||h|||_{s+2,T}^3\right),
\end{eqnarray}
where $c_{28}$ depends on $|||u_0|||_{s+2}$.
\end{lemma}
Let  $\Pi_{\theta}\in\textbf{C}^{\infty}(\textbf{R})$ such that
$\Pi_{\theta}=0$ for $\theta\leq0$ and $\Pi_{\theta}\longrightarrow I$ for $\theta\longrightarrow\infty$.
we introduce a family of smooth functions $S(\theta')$ with $S(\theta')=0$ for $\theta'\leq0$ and $S(\theta')=1$ for $\theta'\geq1$ (see \cite{Schwartz}).
For $W\in\textbf{C}_T^s$, we define
\begin{eqnarray*}
\Pi_{|x'|-|x|}W(t,x)=S(|x'|-|x|)W(t,x),~~x,x'\in \textbf{R}^n.
\end{eqnarray*}
For $l=0,1,2,\ldots,$, by setting
\begin{eqnarray}\label{E3-6}
|x'|:=N_l=2^l,
\end{eqnarray}
then by (\ref{E2-0}), it is directly to check that
\begin{eqnarray}\label{E3-7}
&&\|\Pi_{|x'|-|x|}W\|_{\textbf{H}^{s_1}}\leq c_{s_1,s_2}N_l^{s_1-s_2}\|W\|_{\textbf{H}^{s_2}},~~\forall~s_1\geq s_2\geq0,\\
&&\|\Pi_{|x'|-|x|}W-W\|_{\textbf{H}^{s_1}}\leq c_{s_1,s_2}N_l^{s_1-s_2}\|W\|_{\textbf{H}^{s_2}},~~\forall~0\leq s_1\leq s_2.\nonumber
\end{eqnarray}
For convenience, we denote $\Pi_{N_l-|x|}$ by $\Pi_{N_l}$. We approximate system (\ref{E3-1R1}) and get
the following approximation system
\begin{eqnarray}\label{E3-13}
\mathcal{G}_{N_l}(W^m):=\partial_{tt}W^m-\varepsilon^{-1}\partial_a(\epsilon^{ac}\epsilon^{bd}\gamma(u_0)_{cd}\partial_bW^m)-\varepsilon^2\Pi_{N_l}\mathcal{F}(W^m).
\end{eqnarray}
The following Lemma is to construct the ``$l$ th'' step approximation solution.
\begin{lemma}
There exist a family of linear maps $\Psi^l:\textbf{C}_T^k\mapsto\textbf{C}_{T}^k$ such that
\begin{eqnarray*}
\Psi^l(W^l)=W^l+h^{l+1},~~(t,x)\in[0,T]\times M,~~l=0,1,\cdots
\end{eqnarray*}
where $W^l=\sum_{i=0}^lh^i$, $h^{l+1}$ is the solution of the initial value problem
\begin{eqnarray*}
&&E^{l}+\mathcal{L}_{\varepsilon}(h^{l+1})=0,\\
&&h^{l+1}(0,x)=0,~~\partial_th^{l+1}(0,x)=0,
\end{eqnarray*}
and $E^{l}$ satisfies
\begin{eqnarray*}
E^{l}=\mathcal{G}_{N_{l}}(W^{l})=\mathcal{G}_{N_{l}}(\Psi^{l-1}\underbrace{\circ\cdots\circ}_{l-1}\Psi^{0}(W^0)).
\end{eqnarray*}
Moreover, for $0\leq s\leq k-2$, it holds
\begin{eqnarray}\label{E3-9}
|||h^{l+1}|||_{s,T}\leq c_{29}|||E^l|||_{s,T}.
\end{eqnarray}
\end{lemma}
\begin{proof}
Assume that a suitable ''$0$th step'' approximation solution of (\ref{E3-1R1}) has been chosen, which is $W^0\neq0$. The ``$l$th step'' approximation solution is denoted by
\begin{eqnarray*}
W^l=\sum_{i=0}^lh^i.
\end{eqnarray*}
Define
\begin{eqnarray}\label{E3-8R0}
E^l=\partial_{tt}W^l-\varepsilon^{-1}\partial_a(\epsilon^{ac}\epsilon^{bd}\gamma(u_0)_{cd}\partial_bW^l)-\varepsilon^2\Pi_{N_{l+1}}\mathcal{F}(W^l).
\end{eqnarray}
Then we plan to find the ``$l$th step'' approximation solution $W^{l+1}$. By (\ref{E3-1R1}), we have
\begin{eqnarray}\label{E3-8}
\mathcal{G}(W^l+h^{l+1})&=&\partial_{tt}(W^l+h^{l+1})-\varepsilon^{-1}\partial_a(\epsilon^{ac}\epsilon^{bd}\gamma(u_0)_{cd}\partial_b(W^l+h^{l+1}))-\varepsilon^2\Pi_{N_{l+1}}\mathcal{F}(W^l+h^{l+1})\nonumber\\
&=&\partial_{tt}W^l-\varepsilon^{-1}\partial_a(\epsilon^{ac}\epsilon^{bd}\gamma(u_0)_{cd}\partial_bW^l)-\varepsilon^2\Pi_{N_{l+1}}\mathcal{F}(W^l)\nonumber\\
&&+\partial_{tt}h^{l+1}-\varepsilon^{-1}\partial_a(\epsilon^{ac}\epsilon^{bd}\gamma(u_0)_{cd}\partial_bh^{l+1})-\varepsilon^2\partial_{W^l}\Pi_{N_{l+1}}\mathcal{F}(W^l)h^{l+1}\nonumber\\
&&-\varepsilon^2\Pi_{N_{l+1}}\left(\mathcal{F}(W^l+h^{l+1})+\mathcal{F}(W^l)+\partial_{W^l}\mathcal{F}(W^l)h^{l+1}\right)\nonumber\\
&=&E^l+\mathcal{L}_{\varepsilon}(h^{l+1})+R(h^{l+1}),
\end{eqnarray}
where
\begin{eqnarray}\label{E3-8R1}
R(h^{l+1})=-\varepsilon^2\Pi_{N_{l+1}}\left(\mathcal{F}(W^l+h^{l+1})+\mathcal{F}(W^l)+\partial_{W^l}\mathcal{F}(W^l)h^{l+1}\right).
\end{eqnarray}
By Theorem 3 in section 2, there exists a solution $h^{l+1}$ of
\begin{eqnarray*}
&&E^l+\mathcal{L}_{\varepsilon}(h^{l+1})=0,\\
&&h^{l+1}(0,x)=0,~~\partial_th^{l+1}(0,x)=0.
\end{eqnarray*}
A similar estimate with (\ref{E2-32}) is derived as
\begin{eqnarray*}
|||h^{l+1}|||_{s,T}\leq c_{29}|||E^l|||_{s,T}.
\end{eqnarray*}
Furthermore, one can know from (\ref{E3-8R0}) and (\ref{E3-8}) that
\begin{eqnarray}\label{E3-10}
E^{l+1}=R(h^{l+1}).
\end{eqnarray}
\end{proof}

For $2\leq s_{0}<\bar{s}<s<\tilde{s}\leq k-1$, set
\begin{eqnarray}\label{E3-11}
&&s_l:=\bar{s}+\frac{s-\bar{s}}{2^l},\\
\label{E3-12}
&&\alpha_{l+1}:=s_l-s_{l+1}=\frac{s-\bar{s}}{2^{l+1}}.
\end{eqnarray}
By (\ref{E3-11})--(\ref{E3-12}), it follows that
\begin{eqnarray*}
s_0>s_1>\ldots>s_l>s_{l+1}>\ldots.
\end{eqnarray*}

\begin{theorem}
System (\ref{E3-1}) with initial data (\ref{E3-1R0})
has a solution
\begin{eqnarray}\label{E3-18}
v^m(t,x)=W^m_{\infty}+v^m_0+v^m_1t,
\end{eqnarray}
where $W^m_{\infty}$ has the form
\begin{eqnarray*}
W^m_{\infty}=\sum_{i=0}^{\infty}h_i\in\textbf{C}_{T}^{\bar{s}}.
\end{eqnarray*}
\end{theorem}
\begin{proof}
The proof is based on the induction. For any $l=0,1,2,\ldots$, we claim that there exists a constant $0<d<1$ such that
\begin{eqnarray}\label{E3-14R0}
&&|||h^{l+1}|||_{s_{l+1},T}<d^{2^{l}}<1,\\
\label{E3-14R1}
&&|||E^{l+1}|||_{s_{l+1}}\leq d^{2^{l+1}},\\
\label{E3-14R2}
&&W^{l+1}\in\textbf{B}_{R,T}^{s^l}.
\end{eqnarray}
We choose a fixed sufficient small $W^0>0$ such that
\begin{eqnarray}\label{E3-16}
|||W^0|||_{s_0}\ll1,~~|||E^0|||_{s_0}\ll1.
\end{eqnarray}
For the case $l=0$, by (\ref{E3-9}), we have
\begin{eqnarray}\label{E3-14RRR}
|||h^{1}|||_{s_{1},T}\leq c_{30}|||E^0|||_{s_{1},T}\leq c_{30}|||E^0|||_{s_{0},T}<1.
\end{eqnarray}
It follows from (\ref{E3-5}), (\ref{E3-7}), (\ref{E3-9}) and (\ref{E3-10}) that
\begin{eqnarray}
\label{E3-15RRR}
|||E^{1}|||_{s_{1},T}&\leq&|||R(h^{1})|||_{s_{1},T}\leq c_{28}\varepsilon^2N_1^2\left(|||h^1|||_{s_1,T}^2(1+|||W^0|||_{s_1,T})+|||h^1|||_{s_1,T}^3\right)\nonumber\\
&\leq& c_{31}\varepsilon^2N_{1}^{2}|||h^{1}|||_{s_{1},T}^2\leq c(2\varepsilon|||E^0|||_{s_{0},T})^2.
\end{eqnarray}
It is obviously to see that (\ref{E3-14RRR})-(\ref{E3-15RRR}) gives $(\ref{E3-14R0})_{l=0}-(\ref{E3-14R1})_{l=0}$ by choosing suitable small $\varepsilon>0$. So we get $W^1\in\textbf{B}_{R,T}^{s_1}$.

Assume that (\ref{E3-14R0})-(\ref{E3-14R2}) holds for $1\leq i\leq l$, i.e.
\begin{eqnarray}\label{E3-17}
&&|||h^{i}|||_{s_{i},T}<1,\\
\label{E3-17R0}
&&|||E^{i}|||_{s_{i}}\leq d^{2^{i}},\\
\label{E3-17R1}
&&W^{i}\in\textbf{B}_{R,T}^{s^i}.
\end{eqnarray}
Now we prove that (\ref{E3-14R0})-(\ref{E3-14R2}) holds for $l+1$.
From (\ref{E3-9}) and (\ref{E3-17R0}), we have
\begin{eqnarray}\label{E3-14}
|||h^{l+1}|||_{s_{l+1},T}\leq c_{31}|||E^l|||_{s_{l},T}<c_{31} d^{2^{i}}<1.
\end{eqnarray}
It follows from  (\ref{E3-5}), (\ref{E3-7}), (\ref{E3-9}), (\ref{E3-10}), (\ref{E3-17R1}) and (\ref{E3-14}) that
\begin{eqnarray}
\label{E3-15}
|||E^{l+1}|||_{s_{l+1},T}&\leq&|||R(h^{l+1})|||_{s_{l+1},T}\nonumber\\
&\leq&c_{28}\varepsilon^2N_{l+1}^2\left(|||h^{l+1}|||_{s_{l+1},T}^2(1+|||W^l|||_{s_{l+1},T})+|||h^{l+1}|||_{s_{l+1},T}^3\right)\nonumber\\
&\leq&c_{32}\varepsilon^2N_{l+1}^{2}|||h^{l+1}|||_{s_{l+1},T}^2\nonumber\\
&\leq&c_{33}\varepsilon^2N_{l+1}^{2}|||E^l|||_{s_{l},T}^2\nonumber\\
&\leq&c_{33}\varepsilon^{2+2^2}N_{l+1}^{2}N_l^{2^2}|||E^{l-1}|||_{s_{l-1},T}^{2^2}\nonumber\\
&\leq&\ldots\nonumber\\
&\leq&c_{34}(16\varepsilon|||E^{0}|||_{s_{0},T})^{2^{l+1}}.
\end{eqnarray}
We can choose a fixed sufficient small $\varepsilon>0$ such that
\begin{eqnarray*}
0<16\varepsilon|||E^{0}|||_{s_{0},T}<1.
\end{eqnarray*}
Thus we conclude that (\ref{E3-14R0})--(\ref{E3-14R1}) holds. Note that $W^l=\sum_{i=0}^lh^i$. So (\ref{E3-14R0}) gives (\ref{E3-14R2}).

Therefore, we derive
\begin{eqnarray*}
\lim_{l\longrightarrow\infty}|||E^l|||_{l,T}=0,
\end{eqnarray*}
which implies that system (\ref{E3-1R1}) with zero initial data has a solution
\begin{eqnarray*}
W_{\infty}=\sum_{i=0}^{\infty}h_i\in\textbf{C}_{T}^{\bar{s}}.
\end{eqnarray*}
At last, by (\ref{E3-1R2}) we obtain the solution of system (\ref{E3-1}) with initial data (\ref{E3-1R0}) has a solution
\begin{eqnarray*}
v^m(t,x)=W^m_{\infty}+v^m_0+v^m_1t.
\end{eqnarray*}
\end{proof}

In what follows, we prove that the uniqueness of solution for system (\ref{E3-1}) with initial data (\ref{E3-1R0}).
Assume that there exists anther solution
\begin{eqnarray}\label{E3-20}
\tilde{v}^m(t,x)=\tilde{W}^m_{\infty}+v^m_0+v^m_1t.
\end{eqnarray}
We intend to prove the following result:
\begin{theorem}
Assume that there exists anther solution (\ref{E3-20}) of system (\ref{E3-1}) with initial data (\ref{E3-1R0}) in $\textbf{B}_{R,T}^{\bar{s}}$.
Then $v^m(t,x)\equiv\tilde{v}^m(t,x)$ holds.
\end{theorem}
\begin{proof}
Let
\begin{eqnarray*}
\bar{W}^m_{\infty}=W^m_{\infty}-\tilde{W}^m_{\infty}.
\end{eqnarray*}
We plan to prove that the following initial problem
\begin{eqnarray}\label{E3-21}
&&\partial_{tt}\bar{W}_{\infty}^m-\varepsilon^{-1}\partial_a(\epsilon^{ac}\epsilon^{bd}\gamma(u_0)_{cd}\partial_b\bar{W}_{\infty}^m)-\varepsilon^2(\mathcal{F}(W_{\infty}^m)-\mathcal{F}(\tilde{W}_{\infty}^m))=0,\\
&&\bar{W}_{\infty}^m(0,x)=0,~~\partial_t\bar{W}_{\infty}^m(0,x)=0\nonumber
\end{eqnarray}
has a solution $\bar{W}_{\infty}^m\equiv0$.

Consider the approximation system of (\ref{E3-21}) as
\begin{eqnarray}\label{E3-22}
\mathcal{G}'(\bar{W}_{\infty}^m):=\partial_{tt}\bar{W}_{\infty}^m&-&\varepsilon^{-1}\partial_a(\epsilon^{ac}\epsilon^{bd}\gamma(u_0)_{cd}\partial_b\bar{W}_{\infty}^m)\nonumber\\
&&-\varepsilon^2\Pi_{N_l}(\mathcal{F}(W_{\infty}^m)-\mathcal{F}(\tilde{W}_{\infty}^m))=0.~~~~
\end{eqnarray}
Then using the similar computation process with (\ref{E3-8}), we have
\begin{eqnarray}\label{E3-23}
\mathcal{L}_{\varepsilon}(\bar{W}_{\infty}^m)+E'(t,x)+R(\bar{W}_{\infty}^m)-E'(t,x)=0,
\end{eqnarray}
where $E'(t,x)$ is a function which does not depends on $\bar{W}_{\infty}^m$,
\begin{eqnarray*}
&&\mathcal{L}_{\varepsilon}(\bar{W}_{\infty}^m)=\partial_{tt}\bar{W}_{\infty}^m-\varepsilon^{-1}\partial_a(\epsilon^{ac}\epsilon^{bd}\gamma(u_0)_{cd}\partial_b\bar{W}_{\infty}^m)-\varepsilon^2\Pi_{N_l}\partial_{\tilde{W}_{\infty}^m}\mathcal{F}(\tilde{W}_{\infty}^m)\bar{W}_{\infty}^m,\\
&&R(\bar{W}_{\infty}^m)=\varepsilon^2(\mathcal{F}(W_{\infty}^m)-\mathcal{F}(\tilde{W}_{\infty}^m)-\partial_{\tilde{W}_{\infty}^m}\mathcal{F}(\tilde{W}_{\infty}^m)\bar{W}_{\infty}^m).
\end{eqnarray*}
By Theorem 3 in section 2, there exists a solution $\bar{W}_{\infty}^m$ of
\begin{eqnarray*}
&&\mathcal{L}_{\varepsilon}(\bar{W}_{\infty}^m)+E'(t,x)=0,\\
&&\bar{W}_{\infty}^m(0,x)=0,~~\partial_t\bar{W}_{\infty}^m(0,x)=0.
\end{eqnarray*}
A similar estimate with (\ref{E2-32}) is derived as
\begin{eqnarray*}
|||\bar{W}_{\infty}^m|||_{s,T}\leq c_{35}|||E'|||_{s,T}.
\end{eqnarray*}
Then by (\ref{E3-5}) and (\ref{E3-23}), we have
\begin{eqnarray*}
|||\bar{W}_{\infty}^m|||_{s_l,T}&\leq& c_{35}|||E'|||_{s_l,T}\leq|||R(\bar{W}_{\infty}^m)|||_{s_l,T}\nonumber\\
&\leq&c_{36}\varepsilon^2N_{l}^2(|||\bar{W}_{\infty}^m|||_{s_l,T}^2(1+|||W_{\infty}^m|||_{s_l,T})+|||\bar{W}_{\infty}^m|||_{s_l,T}^3)\nonumber\\
&\leq&c_{37}\varepsilon^2N_{l}^2|||\bar{W}_{\infty}^m|||_{s_{l-1},T}^2\nonumber\\
&\leq&c_{38}\varepsilon^2N_{l}^2N_{l-1}^{2^2}|||\bar{W}_{\infty}^m|||_{s_{l-2},T}^{2^2}\nonumber\\
&\leq&\ldots\nonumber\\
&\leq&c_{39}(8\varepsilon|||\bar{W}_{\infty}^m|||_{s_{0},T})^{2^l}.
\end{eqnarray*}
Choosing a suitable small $\varepsilon$ such that
\begin{eqnarray*}
0<8\varepsilon|||\bar{W}_{\infty}^m|||_{s_{0},T}<1.
\end{eqnarray*}
Thus we obtain
\begin{eqnarray*}
\lim_{l\longrightarrow\infty}|||\bar{W}_{\infty}^m|||_{s_l,T}=0.
\end{eqnarray*}
This completes the proof.
\end{proof}




\end{document}